\documentclass{hjm1}
\usepackage{verbatim}

\usepackage[latin1]{inputenc}
\usepackage{amsmath}
\usepackage{amssymb}
\usepackage{amsfonts}
\usepackage{latexsym}

\newcommand{\ad}{\mathcal A}
\newcommand{\esp}{\vspace*{0.5cm}}
\newcommand{\esple}{\vspace*{0.3cm}}
\newcommand{\s}{\subseteq}

\newcommand{\meni}{\leqslant}
\newcommand{\bc}{\begin{center}}
\newcommand{\ec}{\end{center}}
\newcommand{\pe}{\langle}
\newcommand{\pd}{\rangle}

\newcommand{\cau}{\mathcal{U}}

\newcommand{\vaz}{\emptyset}
\newcommand{\mai}{\geqslant}

\newcommand{\fdl}{\hfill{$\square$}}

\newcommand{\n}{\noindent}

\newcommand\res{{\upharpoonright}}

\newcommand\R{{\mathbb{R}}}

\newcommand\nsa{{\mathfrak n}{\mathfrak s}{\mathfrak a}}
\newcommand\vsa{{\mathfrak v}{\mathfrak s}{\mathfrak a}}

\newcommand\vcp{{\mathfrak v}{\mathfrak c}{\mathfrak p}}
\newcommand\vssa{{\mathfrak v}{\mathfrak s}{\mathfrak s}{\mathfrak a}}
\newcommand\vn{{\mathfrak v}{\mathfrak n}}
\newcommand\nssa{{\mathfrak n}{\mathfrak s}{\mathfrak s}{\mathfrak a}}

\newcommand{\gc}{\mathfrak{c}}
\newcommand{\gp}{\mathfrak{p}}
\newcommand{\ga}{\mathfrak{a}}
\newcommand{\gd}{\mathfrak{d}}
\newcommand{\gb}{\mathfrak{b}}

\newcommand{\zfc}{\mathbf{ZFC}}

\newcommand{\ch}{\mathbf{CH}}

\newcommand{\psia}{\Psi(\mathcal A)}

\newcommand{\ptes}{\mathcal P}

\begin{document}

\title[Selectively $(a)$-spaces under a certain diamond principle]
{Selectively $(a)$-spaces from  \\  almost disjoint families \\ are necessarily countable \\ under a certain parametrized \\ weak diamond principle}
\author[Charles Morgan and Samuel G. da Silva]
{Charles J.G. Morgan and Samuel G. da Silva}
\address{Department of Mathematics, University College London, Gower Street,
London, WC1E 6BT, UK, and\hfill\break
\indent Centro de Matem\'atica e Aplica\c{c}\~oes Fundamentais,
Universidade de Lisboa, Avenida Professor Gama Pinto, 2, 1649-003
Lisboa, Portugal}
\email{charles.morgan@ucl.ac.uk}\
\address{Instituto de Matem\'atica, Universidade Federal da Bahia,
Campus de Ondina, Av. Adhemar de Barros, S/N, Ondina, CEP 40170-110,
Salvador, BA, Brazil}
\email{samuel@ufba.br}
\cyh

\thanks{The second (and corresponding) author's research was supported by PROPCI/UFBA -- Grant PROPI (Edital PROPCI-PROPG 05/2012 -- 
PROPI 2012).}



\begin{abstract}

The second author has recently shown (\cite{recente}) that any selectively $(a)$ almost disjoint family  must have cardinality strictly less than $2^{\aleph_0}$, so under the Continuum Hypothesis such a family is necessarily countable. However, it is also shown in the same paper that $2^{\aleph_0} < 2^{\aleph_1}$ alone does not avoid the existence of uncountable selectively $(a)$  almost disjoint families. We show in this paper that a certain effective parametrized weak  diamond principle is enough to ensure countability of the almost disjoint family in this context. We also discuss the deductive strength of this specific  weak diamond principle  (which is consistent with the negation of the Continuum Hypothesis, apart from other features).

\end{abstract}

\maketitle



\newtheorem{fato}{Fact}
\newtheorem{que}{Question}
\newtheorem{prob}{Problem}



\section{Introduction}

In this paper   we work with a \emph{star selection principle}, namely the property of being a \emph{selectively $(a)$-space}. Star selection principles
combine ideas and techniques from both its constituent parts as topological topics: the \emph{star covering properties} and the \emph{selection principles}. These  topics
were the subject of dozens of papers over the last  years, and have attracted the attention of many strong researchers.
The reader will find background information on star covering properties in the papers \cite{DRRT} and \cite{survey}; for selection principles and topology, we refer to the papers \cite{scheep} and \cite{ljub2}.

Property $(a)$ -- a star covering property --  was introduced by Matveev in \cite{matv}, and its selective version
 was introduced by Caserta, Di Maio and Ko{\v{c}}inac in \cite{ljub}.

\esple

\begin{dfn} [{\cite{matv}}]
A topological space $X$ satisfies  {\bf Property $(a)$} (or is said to be
an {\bf $(a)$-space}) if  for every open cover
$\cau$ of  $X$ and for every dense set  $D \s X$ there is a set $F \s D$ which is closed and
discrete in $X$ and such that  $St(F, \cau) = X$ (where  $St(F, \cau) = \bigcup \{U
\in \cau: U \cap F \neq \vaz \}$).\end{dfn}

\esple

\begin{dfn} [{\cite{ljub}}]
A topological space $X$ is said to be a {\bf selectively $(a)$-space} if for every sequence $\pe \mathcal{U}_n : n < \omega\pd$ of open covers and for every dense set $D \s X$ there is a sequence $\pe A_n: n < \omega\pd$ of subsets of $D$ which are closed and discrete
in $X$ and such that $\{St(A_n, \mathcal{U}_n): n < \omega\}$ covers $X$. \end{dfn}

\esple

Notice that $(a)$ implies selectively $(a)$. The Deleted Tychonoff plank is an example of a selectively $(a)$-space which does not satisfy Property $(a)$ (Example 2.6 of \cite{song}).

In \cite{recente}, the second author has investigated the presence of the selective version of property $(a)$ in a certain class of topological spaces,
the  \emph{Mr\'owka-Isbell spaces from almost disjoint families}, 
and, as expected, many aspects of such presence within this class have combinatorial characterizations or, at least, are closely related to combinatorial and set-theoretical hypotheses. Let us recall how such spaces are constructed. A set $\ad$ of infinite subsets of the set $\omega$ of all natural numbers  is said to be an \emph{almost disjoint }
(or \emph{a.d.}) \emph{family} if for every pair of distinct $X$, $Y \in \ad$  its intersection  $X \cap Y$ is finite. 
For any a.d. family $\ad$ one may
consider the corresponding $\Psi$-\emph{space}, $\psia$, whose underlying set
is given by $\ad \cup \omega$. The points in $\omega$ are declared
isolated and the basic neighbourhoods of a point $A \in \ad$ are given
by the sets $\{A\} \cup (A \setminus F)$ --  for  $F$ varying on the family of all finite subsets of $\omega$. Notice that  $\omega$ is a dense set of isolated points, $\ad$ is
a closed and discrete subset of $\psia$ and basic neighbourhoods of points in  $\ad$ are compact. 
Such spaces are precisely characterized by a specific list of  topological properties; more precisely,  any Hausdorff, first countable, locally compact separable space whose set of non-isolated points is non-empty and
discrete is homeomorphic to a $\Psi$-space (see \cite{vD},
p.154).

Throughout this paper, $\ad$ will always denote an a.d. family of infinite subsets of $\omega$.  $\ad$ will be said to be a \emph{selectively $(a)$ a.d. family} if the corresponding space $\psia$ is selectively $(a)$. In general, for any topological property $\ptes$ we will say that $\ad$ satisfies $\ptes$ in the case of $\psia$ satisfying $\ptes$.

\esple

The following theorem is a general result on selectively $(a)$-spaces (not only for those from almost disjoint families) and was established by
the second author in \cite{recente}.  Recall that the {\em density} of a topological space $X$, $d(X)$, is the minimum of the cardinalities of all dense subsets of $X$, provided this is an infinite cardinal, or is $\omega = \aleph_0$ otherwise.

\begin{thm} [{\cite{recente}}] \label{matveevseletivo} If $X$ is a selectively $(a)$-space and $H$ is a closed and discrete
subset of $X$, then   $|H| < 2^{d(X)}$.
\end{thm}

As an immediate consequence for $\Psi$-spaces, if $\psia$ is a selectively $(a)$-space then $|\ad| < \gc$. It follows that under the Continuum Hypothesis selectively $(a)$ a.d. families are necessarily countable.

However, it is also shown in \cite{recente} that $2^{\aleph_0} < 2^{\aleph_1}$ alone does not avoid the existence of uncountable selectively $(a)$  almost disjoint families. Our main goal  in this paper is to show that a certain effective parametrized weak  diamond principle is enough to ensure countability of the almost disjoint family in this  context.

Our set-theoretical notation and terminology are  standard. In what follows, $\omega = \aleph_0$ denotes the set of all natural numbers (and the least infinite cardinal). $[\omega]^{\omega}$ and $[\omega]^{< \omega}$
denote, respectively, the family of all infinite subsets of $\omega$ and
the family of all finite subsets of $\omega$. The first uncountable cardinal is denoted by $\omega_1 = \aleph_1$. 
For a given set $X$, $|X|$ denotes the
cardinality of $X$. $\mathbf{CH}$ denotes the \emph{Continuum Hypothesis}, which is the statement ``$\gc = \aleph_1$", where $\gc$ is the \emph{cardinality of the continuum}, i.e., $\gc = |\mathbb{R}| = 2^{\aleph_0}$. The \emph{Generalized Continuum Hypothesis} (denoted by $\mathbf{GCH}$) is the statement
``$\aleph_{\alpha + 1} = 2^{\aleph_\alpha}$ for every ordinal $\alpha$''.
A \emph{stationary} subset of $\omega_1$ is a subset of $\omega_1$ which intersects all \emph{club} (closed, unbounded) subsets of $\omega_1$ (where ``closed'' means ``closed in the order topology''). 
\emph{Jensen's Diamond}, denoted by $\diamondsuit$,  is the combinatorial \emph{guessing principle}
asserting the existence of a $\diamondsuit$-sequence, which is a sequence $\pe A_\alpha: \alpha < \omega_1 \pd$ such that $(i)$ $A_\alpha \s \alpha$ for every $\alpha < \omega_1$;  and $(ii)$ the following property holds: for any given set $A \s \omega_1$, the $\diamondsuit$-sequence ``guesses'' $A$ stationarily many times, meaning that $\{\alpha < \omega_1: A \cap \alpha = A_\alpha\}$ is stationary. It is easy to see that $\diamondsuit \rightarrow \ch \rightarrow 2^{\aleph_0} < 2^{\aleph_1}$.

For small uncountable cardinals like $\ga$,  $\gb$, $\gp$ and $\gd$, see \cite{vD}.

Let us describe the organization of this paper. In Section 2 we discuss the deductive strength of a certain effective combinatorial principle,
 namely the principle $\diamondsuit(^\omega \omega,<)$. In order to do that, we present a quick review of the  parametrized weak diamond principles of Moore, Hru\v s\'ak and D\v zamonja (\cite{mhd}). The treatment we give is very similar to the one we have done  for $\diamondsuit(\omega, <)$ in \cite{tchecos}. 
 In Section 3 we present our main theorem: the principle  $\diamondsuit(^\omega \omega,<)$ implies that 
selectively $(a)$ a.d. families are necessarily countable. In Section 4 we justify why all functions in this paper are Borel and present some routine verifications. In Section 5 we present some notes, problems and questions on certain cardinal invariants related to the subject. 

\esp

\section{The effective combinatorial principle $\diamondsuit(^\omega \omega, <)$}

\esp

In  this paper we work with specific instances of \emph{effective} (meaning, Borel) versions of certain combinatorial principles, the so-called 
\emph{parametri\-zed weak diamond principles}
introduced by Moore, Hru\v s\'ak and D\v zamonja in \cite{mhd}.

The family of parameters for the weak diamond principles of \cite{mhd} is given by the category 
$\mathcal{PV}$. This category (named after de Paiva (\cite{val1}) and
Vojt{\'a}{\v{s}} (\cite{voj}), its introducers)   has proven itself useful in several (and quite distant between each other) fields as: linear logic; the study of cardinal 
invariants of the continuum; and complexity theory (see Blass' survey \cite{bla} for more information on this category and its surprising applications). $\mathcal{PV}$ is a small subcategory of the dual of the simplest example of a Dialectica category, 
$\hbox{\bf Dial}_2(\hbox{\bf Sets})$(\cite{val2}).

The {\it objects\/} of $\mathcal{PV}$ are triples $o=(A,B,E)$ consisting of 
sets $A$ and $B$, both of size not larger than $\gc$, and a relation $E\subseteq A\times B$ such that

\bc 

$\forall a\in A\;\;\exists b\in B\; \; a\,E\,b\hbox{ and }\; 
\forall b\in B\;\;\exists a\in A\;\; \neg\, a\,E\,b.$ \ec

$(\phi,\psi)$ is a {\it morphism\/} from $o_2$ $=(A_2,B_2,E_2)$, to 
$o_1$ $=(A_1,B_1,E_1)$, if 
$\phi:A_1\rightarrow A_2$, $\psi:B_2\rightarrow B_1$ and 
$$\forall a\in A_1\;\;\forall b\in B_2 \;\; \phi(a)\, E_2\, b \;\rightarrow\; 
a\, E_1\, \psi(b).$$

The category is partially ordered in the following way: 
$o_1\meni_{GT} o_2$ if there is a morphism from $o_2$ to $o_1$. 
Two objects are Galois-Tukey equivalent, $o_1\sim_{GT} o_2$, if 
$o_1\meni_{GT} o_2$ and $o_2\meni_{GT} o_1$.

Given $o = (A,B,E) \in \mathcal{PV}$, the associated  parametrized weak diamond principle $\Phi(A,B,E)$ corresponds to the following combinatorial 
statement (a typical guessing principle)
(\cite{mhd}):

\esple

``For every function $F$ with values in $A$, defined on the binary tree of height $\omega_1$, there is a function $g: \omega_1 \rightarrow B$ such that $g$ `guesses' every branch of the tree, meaning that for all $f \in$  $^{\omega_1} 2$ the set given by $\{\alpha < \omega_1: F(f \res \alpha) E g(\alpha) \}$ is stationary."

\esple

The function $g$ is sometimes called ``an \emph{oracle} for $F$, given by the principle $\Phi(A,B,E)$''.

In  case of $A = B$, we denote $\Phi(A,B,E)$ as $\Phi(A,E)$.

\esple

The following facts, some of them immediate consequences of the definitions,  are either proved or referred to a proof in \cite{mhd}:

\begin{itemize}

\item $(\R,\R,\ne)$ and  $(\R,\R,=)$ are, respectively,  minimal
and maximal elements of $\mathcal{PV}$ with respect to  $\meni_{GT}$.
\esple

\item If $o_1\meni_{GT} o_2$ then $\Phi(o_2)$
implies $\Phi(o_1)$. So, if one has  $o_1\sim_{GT}
o_2$ then it follows that  $\Phi(o_1)\longleftrightarrow \Phi(o_2)$.
\esple

\item  $\diamondsuit  \leftrightarrow \; \Phi(\R,=)$.
\esple

\item  $\Phi(\R,\neq ) \leftrightarrow \Phi(2,\neq)$. (Abraham, unpublished)
\esple

\item $\Phi(2,\neq) \leftrightarrow \Phi(2,=)$. (this one is obvious)
\esple

\item  $\Phi(2,=) \leftrightarrow 2^{\aleph_0} < 2^{\aleph_1}$ (Devlin, Shelah (\cite{devshe}))

\end{itemize}

\esple

As consequences of the listed facts, notice that for any object $o \in \mathcal{PV}$ the following implications hold:

\bc $\diamondsuit \rightarrow \Phi(o) \rightarrow  2^{\aleph_0} < 2^{\aleph_1}$. \ec

The preceding implications justify the terminology ``weak diamond"\, for these guessing principles -- the cardinal inequality $2^{\aleph_0} < 2^{\aleph_1}$ being the weakest diamond of all. 

Now we turn our interest to effective versions of parametrized weak diamond principles; such effective versions are much more flexible (in the sense that they may hold in much more models, including models of  $2^{\aleph_0} = 2^{\aleph_1}$). Recall that a \emph{Polish space} is a
separable and completely metrizable topological space. A subset of a Polish space is \emph{Borel} if it belongs to the smallest $\sigma$-algebra containing all open subsets of the Polish space.

\esple

\begin{dfn} \label{borel}

$(i)$ An object $o = (A,B,E)$ in $\mathcal{PV}$ is {\bf Borel\/} if $A$, $B$ and $E$ 
are Borel subsets of some Polish space. 

\esple

$(ii)$ A map $f:X\longrightarrow Y$ from a Borel subset of a Polish space to a
Borel subset of another is itself\,  {\bf Borel\/} if for every Borel 
$Z\subseteq Y$ one has that $f^{-1}[Z] \s X$ is Borel. 

\esple

$(iii)$ If $o_1$ and $o_2$ are both Borel then $o_1\meni^B_{GT} o_2$ if there is a
morphism from $o_2$ to $o_1$ with both of its constituent maps Borel,
and $o_1\sim^B_{GT} o_2$ if $o_1\meni^B_{GT} o_2$ and $o_2\meni^B_{GT} o_1$. 

\esple

$(iv)$ A map $F:\,^{<\omega_1}2\longrightarrow A$ is {\bf Borel\/} if it is
level-by-level Borel: {\it i.e.\/}, if for each $\alpha<\omega_1$ the
map $F\res ^{\alpha}2:\,^{\alpha}2\longrightarrow A$ is Borel. 

\esple

$(v)$ If $o$ is Borel we define the principle $\diamondsuit(o)$ as in \cite{mhd}:

\esple

\hspace*{2.5cm} $\forall \hbox{ Borel } F:^{<\omega_1}2\rightarrow A\;\;\exists g\in {}^{\omega_1}B\;\;
\forall f\in {}^{\omega_1}2$

\hspace*{2.5cm} $\{\alpha<\omega_1: F(f\res \alpha) \,E
\, g(\alpha)\}\hbox{ is stationary in }\omega_1.$

\end{dfn}
\esple

As expected, for Borel parametrized diamond principles we also have that if $o_1$, $o_2$ are both Borel and $o_1\meni^B_{GT} o_2$ then
$\diamondsuit(o_2)\longrightarrow\diamondsuit(o_1)$, and, consequently, if $o_1\sim^B_{GT}
o_2$ we have $\diamondsuit(o_2)\longleftrightarrow\diamondsuit(o_1)$. 

As an application of the remark of the preceding paragraph, one has that $\diamondsuit$ is, in fact, equivalent to the effective principle
$\diamondsuit(\R,=)$ (Proposition 4.5 of \cite{mhd}); one has just to check that the constituent maps of the known morphisms (in both
directions)  are 
Borel. We will also use the remark of the preceding paragraph to  show, in a little while, that 

\bc $\diamondsuit(^\omega \omega,<) \longleftrightarrow \diamondsuit(^\omega (\ptes (\omega)),$  $^\omega \omega, E)$, \ec

\n but we have to first define precisely these objects.

\esple

\begin{dfn} $(i)$ The object $(^\omega \omega,$ $^\omega \omega\,, <)$ is defined in the following way: for every $f, g \in$ $^\omega \omega$,
$f < g$ if, and only if, $f(n) < g(n)$ for every $n < \omega$.

$(ii)$ The object $(^\omega (\ptes (\omega)),$  $^\omega \omega, E)$ is defined in the following way: for every $\xi \in$ $^\omega(\ptes(\omega))$ and for every $g \in$ $^\omega \omega$, $\xi E g$ if, and only if, 

\bc $[ \exists m < \omega \,\, (\xi(m) = |\aleph_0|)] \vee [\forall n < \omega\,\, (\xi(n) \subsetneq g(n))]$ \ec
\end{dfn}

\esple

\begin{prop} \label{equivalencia} $\diamondsuit(^\omega \omega,<) \longleftrightarrow \diamondsuit(^\omega (\ptes (\omega)),$  $^\omega \omega, E)$.
\end{prop}

\begin{proof} It suffices to show that $(^\omega \omega,$ $^\omega \omega\,, <) \sim^B_{GT} \,\,(^\omega (\ptes (\omega)),$  $^\omega \omega, E)$, i.e., we have to exhibit morphisms between those objects, in both directions, with all constituent maps Borel. 

\esple

\n \emph{Proof of} $(^\omega \omega,$$ ^\omega \omega, <) \meni^B_{GT} (^\omega (\ptes (\omega)),$  $^\omega \omega, E)$: Let $\phi$  be the inclusion map  

\bc $i:$$\,\,^\omega \omega \rightarrow$$ \,\, ^\omega (\ptes (\omega))$; \ec

\n recall that a sequence of natural numbers is also a sequence of subsets of $\omega$. Let $\psi$  be the identity map. If $\phi(f) E g$ then $f(n)$ is a proper subset of $g(n)$ for every $n < \omega$, and therefore $f < g$.

\esple

\n \emph{Proof of} $(^\omega (\ptes (\omega)),$  $^\omega \omega, E) \meni^B_{GT} \,(^\omega \omega,$$^\omega \omega, <)$: Let $\phi:$ $\, ^\omega (\ptes (\omega)) \rightarrow$$\,^\omega \omega$ be defined as follows: for every sequence  $\xi$ of subsets of $\omega$, say $\xi = \pe \xi(n): n < \omega\pd$, let $\phi(\xi): \omega \rightarrow \omega$ be such that, for every $m < \omega$,

\begin{center}

$ \phi(\xi)(m)= 
\left\{ 
\begin{array}{ll} 
\textrm{max}(\xi(m)) + 1 & \mbox {if the set $\xi(m)$ is finite; and} \\ 
0 & \mbox {otherwise.} \end{array} \right. $

\end{center}

Let $\psi:$$\,\,^\omega \omega \rightarrow$$ \,\, ^\omega \omega$ be the identity map. If $\phi(\xi) < g$ and there is $m < \omega$ such that $\xi(m)$ is infinite, then we have $\xi E g$ as desired. Otherwise,  for every $n < \omega$ the set $\xi(n)$ is finite and $\xi(n) \subsetneq g(n)$, and in this case we also have $\xi E g$. \end{proof}

\esple

We close this section by discussing the deductive strength of the effective diamond principle $\diamondsuit(^\omega \omega,<)$. 
We present a number of results -- and these results turn out to be very similar 
to those  we proved in \cite{tchecos} for $\diamondsuit(\omega,<)$.
First of all, we point out that the objects $(\omega,\omega,<)$ and $(^\omega \omega,$ $^\omega \omega, <)$ are comparable in the category order, with a morphism constituted of Borel maps; this is Lemma 3.4 of \cite{tchecos}. For the sake of completeness, we repeat below the proof:

\esple

\begin{fato}$ (\omega,\omega,<) \meni^B_{GT} (^\omega \omega,$ $ ^\omega \omega, <)$

\end{fato}

\esple

\begin{proof} Let $\phi: \omega \rightarrow$ $^\omega \omega$  be such that, for every $n < \omega$,  $\phi(n)$ is the constant function of value $n$, and $\psi:$ $ ^\omega \omega \rightarrow \omega$ be such that, for every $f: \omega \rightarrow  \omega$, $\psi(f) = min(im(f))$. If $\phi(n) < g$ then all values of $g$ are greater than $n$ and therefore $n < \psi(g)$. \end{proof}

\esple

If follows that $\diamondsuit(^\omega \omega, <) \rightarrow \diamondsuit(\omega,<)$. In Corollary 3.8 of \cite{tchecos} 
it is shown (using results on \emph{uniformizing colourings of ladder systems}, \cite{ladder}) that $\ch$ does not imply $\diamondsuit(\omega, <)$, and therefore we have the  following:

\esple

\begin{fato} $\ch$ does not imply $\diamondsuit(^\omega\omega, <)$. \fdl

\end{fato}

\esple

In other words, $\diamondsuit(^\omega\omega, <)$ is independent of $\ch$ (notice that $\ch$ together with  all weak diamonds hold in models of $\diamondsuit$, for instance under the Axiom of Constructibility). It is a little more complicated to show that the same happens in the other way round, i.e., to show that $\ch$ is independent of $\diamondsuit(^\omega\omega, <)$. 
For this, we will need a very powerful result of \cite{mhd}: in Theorem 6.6 of the referred paper, the authors have shown that $\diamondsuit(A,B,E)$ holds for a large number of models of $\pe A, B, E \pd = \aleph_1$, where $\pe A,B,E \pd$ denotes the  \emph{evaluation} of $(A,B,E)$, defined as

\bc
$\pe A,B,E \pd = \textrm{min}\{|X|: X \s B \,\,\textrm{and} \,\, \forall a \in A \,\,  \exists b \in X \,\, [a E b]\}$
\ec

More precisely, it is shown in Theorem 6.6 of \cite{mhd} that a countable support iteration of length $\omega_2$ of certain forcings (which are compositions of  \emph{Borel partial orders}) forces $\diamondsuit(A,B,E)$ to hold if, and only if, it forces $\pe A,B,E \pd \meni \aleph_1$. An iteration of Sacks forcings satisfies the  hypothesis of this theorem, and it is well-known  that $\gd = \aleph_1$ in the iterated Sacks model (see, e.g., \cite{blahand}). Clearly,   the evaluation of $(^\omega \omega,$$^\omega \omega, <)$ is the dominating number $\gd$. Putting all pieces together, it  follows from all referred results of \cite{mhd} that in the countable support  iteration of length   $\omega_2$ of   Sacks forcings one has $2^{\aleph_0} = 2^{\aleph_1} = \aleph_2$ and $\diamondsuit(^\omega\omega,<)$ holds.

It is also possible to exhibit a model of $\diamondsuit(^\omega \omega, <)$ in which the ``weak diamond'' $2^{\aleph_0} < 2^{\aleph_1}$ is valid; for that, it suffices to proceed as in Proposition 3.10 of \cite{tchecos} (which is a simple  modification of Proposition 6.1 of \cite{mhd}) and get a Cohen model where $\aleph_{\omega_1}$ Cohen reals are added to a model of $\mathbf{GCH}$. Summing up, we have the following:

\esple

\begin{fato} $\diamondsuit(^\omega \omega, <)$ is consistent with $\neg \ch$, regardless of the validity of the  weak diamond $2^{\aleph_0} < 2^{\aleph_1}$. \fdl

\end{fato}

\esple

The previous result is quite interesting if one recalls that $\diamondsuit(^\omega \omega, =)$ is, in fact, Jensen's diamond $\diamondsuit$,
as already remarked.

\section{The Main Theorem}

\esple

In this section we prove that, under $\diamondsuit(^\omega \omega\,, <)$, selectively $(a)$-spaces from almost disjoint families are necessarily countable. 

The following combinatorial characterization of the selective version of property $(a)$ for $\Psi$-spaces was recently established by the second author
(\cite{recente}).

\esple

\begin{prop} [{\cite{recente}}] \label{charac} Let $\ad = \{A_\alpha: \alpha < \kappa\}  \s [\omega]^\omega$ be an a. d. family of size $\kappa$. The corresponding space $\psia$ is selectively $(a)$ if, and only if, the following property holds: for every sequence $\pe f_n: n < \omega\pd$ of functions  in $\,^{\kappa}\omega$, there is a sequence $\pe P_n: n < \omega \pd$ of subsets of $\omega$ satisfying both of the following clauses:

\esple

$(i)$ $(\forall n < \omega)(\forall \alpha < \kappa)[ |P_n \cap A_\alpha| < \omega]$

$(ii)$ $(\forall \alpha <  \kappa)(\exists n < \omega)[P_n \cap A_\alpha  \not\s f_n(\alpha) ]$. \fdl

\esple

\end{prop}

In the view of the preceding characterization, the following theorem shows that the effective diamond principle $\diamondsuit(^\omega \omega\,, <)$ indeed avoids the existence
of a selectively $(a)$  a.d. family of size $\aleph_1$ in a very strong way, since, for any given candidate a.d. family $\ad$, 
a sequence $\pe g_n: n < \omega\pd$ of functions  in $^{\omega_1} \omega$ 
	will be exhibited 
	in such a way that, considering any given sequence  $\pe P_n: n < \omega \pd$ of subsets of $\omega$, then either the first or the second clause above will have stationarily many counterexamples.

\esp

\begin{thm} \label{main} $\diamondsuit(^\omega \omega,<)$ implies that for every a.d. family $\ad = \{A_\alpha: \alpha < \omega_1\}$ there is a sequence $\pe g_n: n < \omega\pd$ of functions in 
$^{\omega_1} \omega$ such that, for every sequence $\pe P_n: n < \omega \pd$ of subsets of $\omega$, 

\esple

\begin{align*}
\hbox{either\,\,}&\{\alpha<\omega_1: \exists n < \omega \,\,\hbox {such that} \,\,  [P_n \cap A_{\alpha}\,\hbox{is infinite}]\} \\
\hbox{ or }&\{\alpha<\omega_1: \forall n < \omega \,\, [P_n \cap A_{\alpha}\subseteq g_n(\alpha)]\}
\end{align*}

\esple

\n is a stationary subset of $\omega_1$.

\end{thm}

\esple

\begin{proof} Topologically, $^\omega(\ptes(\omega))$ is the same as $^\omega(^\omega 2)$, which is homeomorphic to $^\omega 2$ -- so we may fix an enumeration $\{ X_f: f \in$ $^\omega 2\}$ of $^\omega(\ptes(\omega))$ such that the bijection $f \mapsto X_f$ is Borel. For any $f \in$ $^\omega 2$, the sequence $X_f$ of subsets of $\omega$ will be denoted as $\pe X_f(n): n < \omega\pd$.

Let $\ad = \{A_\alpha: \alpha < \omega_1\}$ be an a.d. family and  $F:$$^{<\omega_1} 2\rightarrow\,\, ^\omega(\ptes(\omega))$ be defined in such a way that, for every $h \in$ $^{< \omega_1}2$ and for every $n < \omega$,

\esple

\begin{center}

$ F(h)(n)= 
\left\{ 
\begin{array}{ll} 
A_{dom(h)} \cap X_{h \res \omega}(n) & \mbox {if $dom(h) \mai \omega$; and} \\ 
0 & \mbox {otherwise.} \end{array} \right. $

\end{center}

\esple

By Proposition \ref{equivalencia},  our hypothesis $\diamondsuit(^\omega\omega,<)$ is equivalent to the  principle  $\diamondsuit(^\omega (\ptes (\omega)),$  $^\omega \omega, E)$, so we may consider a function $g: \omega_1 \rightarrow$ $^\omega \omega$ which is an oracle for $F$ given by  $\diamondsuit(^\omega (\ptes (\omega)),$  $^\omega \omega, E)$.
We use the oracle $g$ for defining a sequence $\pe g_n: n < \omega\pd$ of functions in $^{\omega_1} \omega$ in the natural way: for every $n < \omega$ and for every $\alpha < \omega_1$, set $g_n(\alpha) = g(\alpha)(n)$.

Let $P = \pe P_n: n < \omega \pd$ be an arbitrary sequence of subsets of $\omega$. If $\{\alpha<\omega_1: \exists n < \omega \,\,\hbox {such that} \,\,  [P_n \cap A_{\alpha}\,\hbox{is infinite}]\}$ is not stationary, its complement
$\{\alpha<\omega_1: \forall  n < \omega \,\,  [P_n \cap A_{\alpha}\,\hbox{is finite}]\}$ includes a club, say $C$. Let $f: \omega_1 \rightarrow 2$ be any function such that $X_{f\res \omega} = P$. As $g$ is an oracle,  the set 

\esple
\bc

$S = \{\alpha < \omega_1: F(f \res \alpha) E g(\alpha)\}$

\ec
\esple

\n is stationary in $\omega_1$. But then the set $S \cap [\omega,\omega_1[$ is also stationary, and notice that

\esple

\bc

$S \cap [\omega,\omega_1[ = \{\omega \meni \alpha < \omega_1: F(f \res \alpha) E g(\alpha)\}$

$= \{\omega \meni \alpha < \omega_1: \pe A_\alpha \cap P_n: n < \omega\pd E g(\alpha)\}$.

\ec

\esple

Combining  the definition of the relation $E$ and  that of the sequence of functions $\pe g_n: n < \omega\pd$, it turns out that the set $S \cap [\omega,\omega_1[$ is given by

\esple

\bc

$\{\omega \meni \alpha < \omega_1: [ \exists m < \omega \,\,(|A_\alpha \cap P_m)| = \aleph_0)] \vee [\forall n < \omega\,\, (A_\alpha \cap P_n \subsetneq g_n(\alpha))]\}$.

\ec

\esple

Finally, recall that there is a club set $C$ included in 

\bc $\{\alpha<\omega_1: \forall  n < \omega \,\,  [P_n \cap A_{\alpha}\,\hbox{is finite}]\}$. \ec

 Then $S \cap [\omega,\omega_1[ \cap C$ is stationary, and the desired conclusion follows by the fact that $S \cap [\omega,\omega_1[ \cap C \s \{\alpha< \omega_1:  \forall n < \omega \,\, [P_n \cap A_{\alpha}\subsetneq g_n(\alpha)]\}$. \end{proof}

\esple

For every space $\psia$ one has $|\psia| = |\ad|$, and it is also clear that if $\ad$ is a selectively $(a)$ a.d. family then the same holds for any $\ad' \s \ad$. Therefore we have the following

\esple

\begin{cor} \label{corolariomain} Selectively $(a)$-spaces from   almost disjoint families are   necessarily countable under $\diamondsuit(^\omega\omega,<)$.\fdl

\end{cor}

It follows that the statement ``all selectively $(a)$ $\Psi$-spaces are countable"\, is consistent with $\neg \mathbf{CH}$.

\section{All functions in this paper are Borel}

\esple

As commented in the first section, all functions in this paper are Borel, with routine verifications. Intuitively, functions between Borel
subsets of Polish spaces  which are explicitly defined
in terms of standard  set theoretical operations as ``unions'', ``intersections'' and ``taking minima'' (or ``maxima" ) are Borel. For the convenience of the readers who
have  never worked in this context (for instance, topologists with no previous interests  in Descriptive Set Theory),  we include here a verification that the most important function of this paper - the function $F$ of Theorem \ref{main} - is Borel.

So, fix $\alpha \mai \omega$ and consider the Polish space \,$^\alpha 2$ (which is homeomorphic to the well-known Cantor set) and let $F_\alpha:\, ^\alpha 2 \rightarrow$ $^\omega(\ptes(\omega))$ be the restriction of $F$ to $2^\alpha$. As Polish spaces are spaces with a countable base, it suffices to check that
the inverse images of subbasic open sets are Borel.

Writing $F$ as a function from $^{< \omega_1} 2$ into $^\omega(^\omega 2)$, the expression of $F(h)$ is as follows: for every $n < \omega$, 

\begin{center}

$ F(h)(n)= 
\left\{ 
\begin{array}{ll} 
\chi(A_{dom(h)} \cap X_{h \res \omega}(n))  & \mbox {if $dom(h) \mai \omega$; and} \\ 
0 & \mbox {otherwise.} \end{array} \right. $

\end{center}

\n where $\chi(Y)$ denotes (of course) the characteristic function of $Y$, whenever $Y \s \omega$.

For any $m < \omega$ and $i < 2$,  $[\{\pe m,i\pd\}]$ denotes the canonical subbasic open set of $2^\omega$ given by $\{f \in$ $^\omega 2: f(m) = i\}$
and, for any $n < \omega$, let $S_{n,[\{\pe m,i\pd\}]}$ denote the canonical subbasic open set of $^\omega(2^\omega)$ given by

\bc

$S_{n,[\{\pe m,i\pd\}]} = \{f \in$ $^\omega(^\omega 2): f(n)(m) = i\}$.

\ec

\esple

It follows that

\begin{align*}
F_\alpha^{-1}[S_{n,[\{\pe m,i\pd\}]}] 
& = \{h \in\,^\alpha 2: F(h) \in S_{n,[\{\pe m,i\pd\}]}\} \\
& = \{h \in\,^\alpha 2: F(h)(n)(m) = i\}.
\end{align*}

In case of $i = 1$, the preceding set is 

\bc 

$\{h \in$ $^\alpha 2: m \in A_\alpha \cap X_{h \res \omega}(n)\}$

\ec

\n and in case of $i = 0$ one has just to replace ``$\in$'' by ``$\notin$''. As these sets are complementary we have just to check that one
of them is Borel.

So, for arbitrary and fixed  natural numbers $m$ and $n$, consider the set given by
$\{h \in$ $^\alpha 2: m \in A_\alpha \cap X_{h \res \omega}(n)\}$. As $A_\alpha$ is fixed since the beginning, this set is empty in the case of $m \notin A_\alpha$. So, our ``real set of interest''(the one we have to really check that it is Borel)  is

\bc

$Y_{m,n} = \{h \in$ $^\alpha 2: m \in X_{h \res \omega}(n)\}$

\ec

\n for arbitrary $m,n < \omega$. But this set may be written as $\xi^{-1}[Z_m]$, where 

\bc $Z_m = \{Y \s \omega: m \in Y\}$ \ec

\n  and $\xi:$ $^\alpha 2 \rightarrow \ptes(\omega)$ is given by $\xi = \zeta \circ \beta \circ \gamma$, where

\begin{itemize}

\item $\gamma:$ $^\alpha 2 \rightarrow$ $^\omega 2$ is the restriction to $\omega$, i.e., $\gamma(h) = h \res \omega$ for every $h \in$ $^\alpha 2$;
$\gamma$ is continuous.

\item $\beta:$ $^\omega 2 \rightarrow$ $^\omega(\ptes(\omega))$ is the Borel bijection fixed in the proof, i.e., $\beta(f) = X_f$ for every $f \in$ $^\omega 2$. Recall that, in fact, $\beta$ could be chosen
 as a homeomorphism. 

\item $\zeta:$ $^\omega(\ptes(\omega)) \rightarrow \ptes(\omega)$ is the continuous function given by $\zeta(s) = s(n)$ for every $s \in$ $^\omega(\ptes(\omega))$.

\end{itemize}

It follows that $\xi$ is Borel. As $Z_m$ is open in $\ptes(\omega)$ (when identified with  $^\omega 2$), the verification is finished.

Of course, for exhausting the process of checking the  conditions of Definition \ref{borel} it is also necessary verifying that the objects $(^\omega \omega,$$ ^\omega \omega, <)$ and  $(^\omega (\ptes (\omega)),$  $^\omega \omega, E)$ are  Borel. These verifications are easier; let us present  the first one.  Notice that we only have to decide, after fixing $m < \omega$, if the set $X_m = \{\pe f,g\pd: f(m) < g(m)\}$ is Borel, since the relation $<$ in $^\omega \omega \times\, ^\omega \omega$ is given by the countable intersection $< \,\, = \bigcap\limits_{m < \omega} X_m$. Note also that, for any $m < \omega$ one has

\bc

$X_m = \bigcup\limits_{k < l < \omega} \{\pe f,g \pd: f(m) = k$  and  $g(m) = l\}$

\ec

\n and therefore it suffices to check that $Y_{k,l} = \{\pe f,g \pd: f(m) = k$  and  $g(m) = l\}$ is Borel for any fixed $k$ and $l$. But this set is, indeed, an   open set, since it is
the pre-image of the isolated point $\pe k,l \pd$ (of the discrete space $\omega \times \omega$)  under  the continuous function $\varphi:$\,\,$^\omega \omega \times$$^\omega \omega \rightarrow \omega \times \omega$ given by
$\varphi(f,g) = \pe f(m), g(m) \pd$. This ensures that the relation $<$ is a $G_\delta$ set, since each $X_m$ is an open set.\footnote{The referee noticed that, in fact, the relation $<$ is closed -- because each $X_m$ is a clopen set. To see this, let $\pe f_n, g_n \pd$ be a sequence in $X_m$ converging to some $\pe f,g \pd$. There is a natural number $n_0$ such that $f_n(m) = f(m)$ and $g_n(m) = g(m)$ for all $n \mai n_0$. Thus necessarily $f(m) < g(m)$ and $\pe f,g \pd \in X_m$.}

\section{Notes,  Questions and Problems}

\esple

As remarked in the first section (see Theorem \ref{matveevseletivo}), there are no 
selectively $(a)$ a.d. families of size $\gc$; on the other hand, countable a.d. families are associated to metrizable $\Psi$-spaces, so if $\ad$ is countable then $\psia$ is paracompact and therefore it is $(a)$ (thus, selectively $(a)$). Considering all, the
 cardinal invariants  we introduce below  are both uncountable and not larger than $\gc$.

\esple

\begin{dfn} The cardinal invariants  $\nssa$ and $\vssa$  are defined in the following way:

\esple

$\nssa = \min\{|\ad|: \ad$ is not selectively $(a) \}$; and

$\vssa = \min\{\kappa:$ if  $|\ad| = \kappa$, then $\ad$ is not selectively $(a)\}$. \fdl

\esple

\end{dfn}

The cardinal $\nssa$ is a ``non'' cardinal invariant and $\vssa$   is a ``never'' cardinal invariant, in the sense of \cite{tchecos} -- where  we have defined $\vsa$ as being the least $\kappa$ such that no $\ad$ of size $\kappa$ is $(a)$ (so,  $\vsa \meni \vssa$), $\vn$ as being the least $\kappa$ such that no $\ad$ of size $\kappa$ is normal and $\vcp$ as the least $\kappa$ such that no $\ad$ of size $\kappa$ is countably paracompact.  $\vn$ and $\vsa$ are well defined, respectively,  because of  Jones' Lemma (\cite{jones})  and Matveev's $(a)$-version of Jones Lemma (\cite{matv}) (in fact, Theorem \ref{matveevseletivo} (from \cite{recente}) is a kind of selective version of Matveev's referred result); and $\vcp$ is well defined because of results due to Fleissner (it is proved in \cite{fle} that countably paracompact 
separable spaces cannot include closed discrete subsets of size $\gc$). 

Let us summarize the known $\zfc$  inequalities involving these cardinals. In what follows, $\nsa$ is the least $\kappa$ such that there is an a.d. family which is not $(a)$ (see  \cite{sze})\footnote{The original definition of $\nsa$ was done in terms of \emph{soft almost disjoint families}, but the
definitions are equivalent since an a.d. family is $(a)$ if and only if every finite modification of $\ad$ is soft. Again, see \cite{sze} for details.}.

\esple

\begin{fato} The following inequalities hold in $\zfc$:

\esple

$(i)$  $\gp \meni \nsa \meni \gb \meni \ga$; 

$(ii)$  $\gb \meni \gd \meni \nssa \meni \vssa$;

$(iii)$  $\vsa \meni \vssa$;  and

$(iv)$ $\vn \meni \vcp$. \fdl

\end{fato}

\esple

Notice that we still have very few information on $\vsa$ and $\vcp$.

About the above inequalities: $\gb \meni \ga$ and $\gb \meni \gd$ are well known (see \cite{vD}). Szeptycki and Vaughan have proved in \cite{szepvau}
that a.d. families of size less than $\gp$ satisfy property $(a)$, so\, $\gp \meni \nsa$. The inequality $\nsa \meni \gb$ is due to works of Brendle, Brendle-Yatabe and Szeptycki (see \cite{sze}). 
And the second author of the present paper proved in \cite{recente} that a.d. families of size strictly less than $\gd$ are selectively $(a)$, and therefore
$\gd \meni \nssa \meni \vssa$. The display of the above inequalities left clear that
if $\gb < \gd$ then $\nsa < \nssa$, i.e., if $\gb < \gd$ then there are selectively $(a)$, non-$(a)$ a.d. families. Notice also that $\gb < \gd$ is consistent with $2^{\aleph_0} < 2^{\aleph_1}$, since both inequalities hold after adding $\aleph_{\omega_1}$ Cohen reals to a model of $\mathbf{GCH}$ (see details in \cite{recente}). $\vn \meni \vcp$ holds because normal $\Psi$-spaces are countably paracompact (\cite{QA}).

Szeptycki and Vaughan have proved  that normal
$\Psi$-spaces of size less than $\gd$ are $(a)$-spaces (\cite{szepvau}). Because of this, if  $\vn \meni \gd$ then one has  $\vn \meni \vsa \meni \vssa$.

In \cite{tchecos}, we have presented a problem (inspired by the upper bounds for $\vn$ in terms of other ``never'' cardinal invariants which are obtained
when assuming $\vn \meni \gd$) of finding upper bounds for $\vsa$, $\vn$ and $\vcp$ in terms of other cardinal invariants such as $\gd$, $\ga$ or $\gb$ (Problem 5.3 of \cite{tchecos}). However, after some time we realized that it is more likely that  those ``never cardinal invariants'' have \emph{lower bounds} given by other cardinal invariants; note that we have just proved $\gd \meni \vssa$\footnote{In a private communication, Michael Hru\v s\'ak also
pointed out to the second author that  these ``never'' cardinals are more likely to have definable lower bounds than upper bounds.}. So, it is better (and probably wiser) to actualize the referred problem.

\esple

\begin{prob} Search for both lower and upper bounds for the ``never'' cardinal invariants, in terms of other known cardinal invariants.\fdl
\end{prob}

\esple

Some results of the present paper can be translated into the language of ``never cardinal invariants''; for instance, Theorem \ref{main} and Corollary \ref{corolariomain} essentially told us the following:

\esple

\begin{thm} $\diamondsuit(^\omega \omega, <)$ implies $\vssa = \aleph_1$.\fdl
\end{thm}

\esple

It is worthwhile remarking that in \cite{recente} the second author has proved that $2^{\aleph_0} < 2^{\aleph_1}$ alone does not imply $\vssa = \aleph_1$ (Proposition 5.2 of \cite{recente}). Recall that, as a consequence of Jones' Lemma, $2^{\aleph_0} < 2^{\aleph_1}$ implies $\vn = \aleph_1$. It is still an open question whether $2^{\aleph_0} < 2^{\aleph_1}$ alone  implies $\vcp = \aleph_1$ or $\vsa = \aleph_1$ (Question 5.4 of \cite{tchecos}).
 The effective parametrized weak diamond principle
$\diamondsuit(\omega,<)$ does imply $\vsa = \vcp = \aleph_1$ (Propositions 4.1 and 4.3 of \cite{tchecos}).

It is asked in \cite{recente} if is it consistent that there is an a.d. family of size $\gd$ such that $\psia$ is selectively $(a)$\footnote{It is consistent that there is an a.d. family of size $\gp$ which is $(a)$ (\cite{szepvau}) and there is a $\zfc$ example of an a.d. family of size $\gb$ which is not $(a)$ (\cite{sze}). Notice that there is no way of proving  within $\zfc$ that the latter  a.d. family is not selectively $(a)$, since $\gb < \gd$ is consistent.}.  The following related question was formulated by Rodrigo Dias during a session of the USP Topology Seminar at S\~ao Paulo.

\esple

\begin{que}\label{nssa} Is there a $\zfc$ example of an a.d. family of size $\gd$ such that $\psia$ {\bf is not} a selectively $(a)$-space ? \fdl
\end{que}

\esple

The previous question is related to the problem of searching upper and lower bounds; if the answer is yes, then $\nssa = \gd$ -- i.e, if the previous question has positive answer then one of the ``non'' cardinals would coincide with $\gd$ in $\zfc$. 

\esple

\begin{que}\label{vssa} $\zfc$ proves $\vssa = \gd$  ? \fdl
\end{que}

\esple

Recall that, as just mentioned, the inequality $\gd \meni \vssa$ holds, so $\gd$ is, at least, a lower bound for $\vssa$. The same does not hold for the other ``never"\,  cardinal invariants: as already remarked, $2^{\aleph_0} < 2^{\aleph_1}$ is enough to ensure that $\vn = \aleph_1$, and $\diamondsuit(\omega,<)$ does the same for $\vsa$ and $\vcp$. Both hypotheses ``$2^{\aleph_0} < 2^{\aleph_1}$"\, and ``$\diamondsuit(\omega,<)$"\, are consistent with $\aleph_1 < \gd$. In fact, the conjunction of the three mentioned statements ($\diamondsuit(\omega,<)$,  $2^{\aleph_0} < 2^{\aleph_1}$ and $\aleph_1 < \gd$) holds in the already mentioned Cohen model of the Proposition
3.10 of \cite{tchecos}.

Notice that a positive answer to Question \ref{vssa} would be a strengthening to a positive answer to Question \ref{nssa}.

\esple

We close this paper with the following problem:

\begin{prob} Search for purely combinatorial, non-trivial equivalent definitions for the ``never'' cardinal invariants. \fdl
\end{prob}

\esple

The same question may be formulated for the ``non'' cardinal $\nssa$, but  it was already settled for the other ``non'' cardinals. It is shown in \cite{sze} that $\nsa = \ga \gp$, where $\ga \gp$ is a 
combinatorially defined cardinal. The ``non normal'' and the ``non countably paracompact'' cardinals are both equal to $\aleph_1$, because of Luzin gaps 
(see \cite {hms}).

\esple

\n {\bf Acknowledgements.} The authors acknowledge the anonymous referee for his (or her) careful reading of the manuscript and for a number of   comments, corrections and suggestions which improved the presentation of the paper.


\begin{thebibliography}{00}



\bibitem [1] {ladder} Balogh, Z., Eisworth, T., Gruenhage, G., Pavlov, O.,
and Szeptycki, P., {\em Uniformization and anti-uniformization
properties of ladder systems}, Fundamenta Mathematicae {\bf 181}, 3 (2004), 189-213.


\bibitem [2] {bla} Blass, A., {\em Questions and answers--a
category arising in linear logic, complex\-ity theory, and set theory.}
Advances in linear logic (Ithaca, NY, 1993), London Mathematical Society Lecture
Notes Series  {\bf 222}, Cambridge Univ. Press, Cambridge, 1995, 61-81.
(arxiv.org/PS\_cache/math/pdnf/9309/9309208v1.pdnf)
	

\bibitem [3] {blahand} Blass, A., {\em Combinatorial cardinal characteristics of the continuum}, In: Foreman, M., Kanamori, A.,  and Magidor, M. (eds), 
 {\em Handbook of set theory. {V}ols. 1, 2, 3}, Springer, Dordrecht, 2010,  395--489. 
    

\bibitem [4] {ljub} Caserta, A., Di Maio, G. and Ko{\v{c}}inac, Lj. D. R., {\em Versions of properties $(a)$ and $(pp)$}, 
Topology and its Applications {\bf 158}, 12 (2011), 1360--1368. 

\bibitem [5] {devshe} Devlin, K. J., and Shelah, S., {\em A weak version
of $\diamondsuit $ which follows from $2\sp{\aleph \sb{0}}<2\sp{\aleph
\sb{1}}$.} Israel J. Math.  {\bf 29}, 2-3 (1978), 239-247. 


\bibitem [6] {vD} van Douwen, E.K.,
    {\em The integers and topology}. In Kunen, K., Vaughan, J. E. (eds.),
{\em Handbook of set-theoretic topology}, pp.111-167,
 North-Holland, Amsterdam, 1984.
 
\bibitem  [7]   {DRRT} van Douwen, E. K.,   Reed, G. M.,  Roscoe, A. W.,  and  Tree, I. J.,  {\em Star covering properties},
 Topology and its Applications {\bf 39}, 1  (1991), 71--103.
 
\bibitem [8] {fle} Fleissner, W. G., {\em Separation properties in Moore spaces}, 
 Fundamenta Mathematicae {\bf  98}, 3  (1978),  279--286.

 


\bibitem [9] {hms} Hru\v s\'ak, M. ;  Morgan, C. J. G. ;  da Silva, S. G., {\em  Luzin gaps are not countably paracompact}, 
 Questions Answers Gen. Topology  {\bf 30}, 1  (2012), 59--66.
		
 
 
\bibitem [10] {jones} Jones, F. B., {\em Concerning normal and completely normal spaces}, Bulletin of the  American Mathematical Society  {\bf 43}, 10 (1937), 671-677.




\bibitem  [11]   {ljub2} Ko{\v{c}}inac, Lj. D. R., {\em Selected results on selection principles},
 Proceedings of the 3rd Seminar on Geometry and Topology, 
 Azarb. Univ. Tarbiat Moallem, Tabriz, Iran, 2004, pp.  71--104.


\bibitem [12] {matv} Matveev, M.V., {\em Some questions on property (a)},
  Questions and Answers in General Topology {\bf 15}, 2 (1997), 103--111.
  
\bibitem [13]    {survey} Matveev, M.V., {\em A survey on star covering properties}, Topology Atlas, Preprint 330, 1998. 


\bibitem [14] {mhd}  Moore, J.T., Hru{\v{s}}{\'a}k, M., and
D{\v{z}}amonja, M., {\em Parametrized $\diamondsuit$ principles},
Trans. Amer. Math. Soc.  {\bf 356}, 6 (2004), 2281--2306.


\bibitem [15]  {tchecos} Morgan, C.J.G., da Silva, S.G., {\em Almost
disjoint families and ``never'' cardinal invariants},
Comment.~Math.~Univ.~Carolin., {\bf 50}, 3  (2009), 433--444.



\bibitem [16] {val1} de Paiva, V. C. V., {\em A Dialectica-like model of
linear logic}, Category theory and computer science (Manchester, 1989),
Lecture Notes in Comput. Sci. {\bf 389} (1989), Springer, Berlin,
341-356.  (www.cs.bham.ac.uk/$\sim$vdp/publications/CTCS89.pdnf)
  
\bibitem [17] {val2} de Paiva, V. C. V., {\em Dialectica and Chu
constructions: cousins?}, Theory and Applications of Categories, Vol.
17, No. 7, 2007, 127-152.  (www.tac.mta.ca/tac/volumes/17/7/17-07.pdnf)


\bibitem  [18]    {scheep} Scheepers, M., {\em Selection principles and covering properties in topology},
 Note di Matematica {\bf 22}, 2  (2003/04),  3--41.
 

\bibitem [19] {QA} da Silva, S.G., {\em On the presence of countable
paracompactness, normality and property $(a)$ in spaces from almost
disjoint families}, Questions and  Answers in General Topology {\bf 25}, 1 (2007),
1--18.

\bibitem [20] {recente} da Silva, S. G., {\em $(a)$-spaces and selectively $(a)$-spaces from almost disjoint families}, Acta Mathematica Hungarica {\bf 142}, 
2 (2014), 420--432.




\bibitem [21] {sze} Szeptycki, P. J., {\em Soft almost disjoint
families}, Proceedings of American Mathematical Society  {\bf 130}, 12 (2002), 3713--3717. 
  
  


\bibitem [22]  {szepvau} Szeptycki, P.J. and Vaughan, J.E.,
    {\em Almost disjoint families and property (a)},
   Fundamenta Mathematicae {\bf 158}, 3 (1998), 229--240.

\bibitem [23] {song}  Song, Y. K. {\em Remarks on selectively (a)-spaces},
 Topology Appl. {\bf  160}, 6  (2013),   806--811.


\bibitem [24] {voj} Vojt{\'a}{\v{s}}, P. { \em Generalized
Galois-Tukey-connections between explicit relations on classical objects
of real analysis}, Set theory of the reals (Ramat Gan, 1991), Israel
Math. Conf. Proc.  {\bf 6}, Bar-Ilan Univ., Ramat Gan (1993), 619-643. 
		



\end{thebibliography}
\end{document}